\documentclass[12pt, a4paper, leqno]{amsart}
\usepackage[utf8]{inputenc}
\usepackage{amsmath}
\usepackage{amsfonts}
\usepackage{amssymb}
\usepackage{times}
\usepackage[T1]{fontenc} 
\usepackage{url} 
\usepackage{color,esint}
\usepackage[dvipsnames]{xcolor}
\usepackage[colorlinks, allcolors=RedViolet,pdfstartview=,pdfpagemode=UseNone]{hyperref} 
\usepackage{graphicx} 
\usepackage{enumitem}
\usepackage{tabularx}
 \usepackage{mathtools}
\usepackage{pgf,tikz}
\usepackage{mathrsfs}
\usetikzlibrary{arrows}

\let\oldmarginpar\marginpar
\renewcommand\marginpar[1]{\-\oldmarginpar[\raggedleft\footnotesize #1]%
{\raggedright\footnotesize #1}}

\usepackage{amsthm}
\theoremstyle{plain}
\newtheorem{thm}[equation]{Theorem}
\newtheorem{lem}[equation]{Lemma}
\newtheorem{prop}[equation]{Proposition}
\newtheorem{cor}[equation]{Corollary}

\theoremstyle{definition}
\newtheorem{defn}[equation]{Definition}

\newtheorem{eg}[equation]{Example}

\theoremstyle{remark}

\numberwithin{equation}{section}

\newcommand{\R}{\mathbb{R}}
\newcommand{\N}{\mathbb{N}}

\newcommand{\BV}{\ensuremath{\mathrm{BV}}}
\newcommand{\V}{\ensuremath{\mathrm{V}}}
\newcommand{\EV}{\ensuremath{\mathrm{EV}}}
\newcommand{\Volim}{\ensuremath{\mathrm{\overline V}}}
\newcommand{\Vulim}{\ensuremath{\mathrm{\underline V}}}
\newcommand{\RBVz}{\ensuremath{\mathrm{RBV}_0}}
\newcommand{\RBV}{\ensuremath{\mathrm{RBV}}}
\newcommand{\AC}{\ensuremath{\mathrm{AC}}}
\newcommand{\Lip}{\ensuremath{\mathrm{Lip}}}

\renewcommand{\phi}{\varphi}
\renewcommand{\epsilon}{\varepsilon}
\def\le{\leqslant}
\def\leq{\leqslant}
\def\ge{\geqslant}
\def\geq{\geqslant}
\def\phi{\varphi}
\def\rho{\varrho}
\def\vartheta{\theta}

\newcommand{\Phiw}{\Phi_{\text{\rm w}}}
\newcommand{\Phic}{\Phi_{\text{\rm c}}}

\newcommand{\inc}[1]{\hyperref[def:aInc]{{\normalfont(Inc){\ensuremath{_{#1}}}}}}
\newcommand{\dec}[1]{\hyperref[def:aDec]{{\normalfont(Dec){\ensuremath{_{#1}}}}}}
\newcommand{\ainc}[1]{\hyperref[def:aInc]{{\normalfont(aInc){\ensuremath{_{#1}}}}}}
\newcommand{\adec}[1]{\hyperref[def:aDec]{{\normalfont(aDec){\ensuremath{_{#1}}}}}}
\newcommand{\adeci}[1]{\hyperref[def:aDeci]{{\normalfont(aDec){\ensuremath{_{#1}^\infty}}}}}
\newcommand{\azero}{\hyperref[def:a0]{{\normalfont(A0)}}}
\newcommand{\aone}{\hyperref[def:a1]{{\normalfont(A1)}}}
\newcommand{\vaone}{\hyperref[def:va1]{{\normalfont(VA1)}}}
\newcommand{\aones}[1]{\hyperref[def:a1s]{{\normalfont(A1-{\ensuremath{{#1}})}}}}

\date{\today}

\begin{document}

\title{Riesz spaces with generalized Orlicz growth}
\author{Peter Hästö, Jonne Juusti and Humberto Rafeiro}


\begin{abstract}
We consider a Riesz $\phi$-variation for functions $f$ defined on the real line when 
$\phi:\Omega\times[0,\infty)\to[0,\infty)$ is a generalized $\Phi$-function. 
We show that it generates a quasi-Banach space and derive an explicit 
formula for the modular when the function $f$ has bounded variation. The resulting $BV$-type energy has 
previously appeared in image restoration models. We 
generalize and improve previous results in the variable exponent and 
Orlicz cases and answer a question regarding the Riesz--Medvedev variation 
by Appell, Bana\'s and Merentes [\emph{Bounded Variation and Around}, 
Studies in Nonlinear Analysis and Applications, Vol.\ 17, De Gruyter, Berlin/Boston, 2014].
\end{abstract}

\maketitle

%
\section{Introduction}

The classical total variation of $f:[a,b]\to \R$, defined as 
\[
\sup_{a=x_1<\cdots<x_n=b} \sum_{k=1}^{n-1} |f(x_{k+1}) - f(x_k)|,
\]
is an intuitive way to measure the variation of a function in one dimension. 
Appell, Bana\'s and Merentes provide many different versions of the variation, including what they call the Riesz--Medvedev variation (\cite[Section~2.4]{AppBM14}, originally from \cite{Med53}):
\[
\sup_{a=x_1<\cdots<x_n=b} \sum_{k=1}^{n-1} \phi \bigg(\frac{|f(x_{k+1}) - f(x_k)|}{|x_{k+1}-x_k|} \bigg) |x_{k+1}-x_k|. 
\]
When $\phi(t)=t$, this reduces to the normal total variation, above. 
In 2016, Castillo, Guzm\'an and Rafeiro \cite{CasGR16} generalized the Riesz--Medvedev variation to the variable exponent case. 
In this article, we further extend and improve their result to the generalized Orlicz case
and answer a question regarding the Riesz--Medvedev variation 
by Appell, Bana\'s and Merentes \cite{AppBM14}. 

Generalized Orlicz spaces, also known as Musielak--Orlicz spaces, have been studied 
with renewed intensity recently \cite{HarH19, LanM19, MizOS20, OhnS21} as have related PDE 
\cite{BenHHK21, BenK20, Bui21, ChlZ_pp, ChlGZ19, HarHL21, KarL_pp, WanLZ19}. 
A contributing factor is that they cover 
both the variable exponent case $\phi(x,t):=t^{p(x)}$ \cite{DieHHR11} and the double phase case $\phi(x,t):=t^p+a(x)t^q$ \cite{BarCM18}, 
as well as their many variants:
perturbed variable exponent, 
Orlicz variable exponent, 
degenerate double phase, 
Orlicz double phase, 
variable exponent double phase, 
triple phase and 
double variable exponent. 
For references see \cite{HasO22}.

Chen, Levine and Rao \cite{CheLR06} proposed a generalized Orlicz model for image restoration with the energy function 
\[
\phi(x, t) := 
\begin{cases}
\frac1{p(x)} t^{p(x)}, & \text{when } t\le 1, \\
t - 1-\frac1{q(x)}, & \text{when } t>1.
\end{cases}
\]
Based on input $u_0$, one seeks to minimize the sum of the regularization and the fidelity term:
\[
\int_\Omega \phi(x, |\nabla u|) + |u-u_0|^2\, dx.
\]
The variable exponent $p$ is chosen to be close to $1$ in areas of potential 
edges in the image and close to $2$ where no edges are expected. This allows us 
to avoid the so-called stair-casing effect whereby artificial edges are introduces in the image 
restoration process. 

A feature of their functional is that $\phi(x,t)\approx t$ for large $t$. Thus Chen, Levine and Rao could use the classical space 
$\BV(\Omega)$ directly. Li, Li and Pi \cite{LiLP10} suggested an image restoration model 
with variable exponent energy restricted away from $1$, so that no $\BV$-spaces are needed.
In \cite{HarH21, HarHLT13}, we considered pure variable exponent and double phase 
versions of this model with $p\to 1$. 
In this case, we cannot use the space $\BV(\Omega)$, and are led to the regularization terms 
\[
\int_\Omega |\nabla f|^{p(x)}\, dx + |D^sf|(\{p=1\}). 
\quad\text{and}\quad
\int_\Omega |\nabla f| + a(x)\,|\nabla f|^q\, dx + |D^sf|(\{a=0\}), 
\]
where $\nabla f$ and $D^sf$ are the absolutely continuous and singular parts of the derivative, respectively. 
Here the singular part of the derivative (i.e.\ the edges in the image) is concentrated on the sets 
$\{p=1\}$ or $\{a=0\}$ and the exponent $p$ or coefficient $a$ should be chosen accordingly. 

The papers \cite{CheLR06, HarH21, HarHLT13} are all based on special structure of $\phi$. 
The problem of defining a $\BV$-type space based on generalized Orlicz growth has been recently 
attacked in \cite{EleHH_pp0, EleHH_pp}. In this paper we show that the Riesz $\phi$-variation 
gives the aforementioned energies in the variable exponent and double phase cases 
(Corollaries~\ref{cor:preciseBVpx} and \ref{cor:preciseBVdp}). This provides support for 
our formulation of the Riesz $\phi$-variation as well as the generalized Orlicz growth 
models of image restoration.

To state some corollaries of the main result (Theorem~\ref{thm:preciseBV}) 
we define some variants of the Riesz $\phi$-variation in generalized Orlicz spaces. For further definitions see the next section.

\begin{defn} \label{def:variation}
Let $I\subset \R$ be a closed interval, $\phi\in \Phiw(I)$ and $f: I\to \R$. We define the functional $\V^\phi_{I}: \mathbb R^I \to [0,\infty]$ by \[
\V^\phi_{I}(f)
:=
\sup_{(I_k)} \V^{\phi} (f, (I_k))
:=
\sup_{(I_k)} \sum_{k=1}^n \phi^+_{I_k} \bigg(\frac{|\Delta_kf|}{|I_k|}  \bigg) |I_k|,
\]
where the supremum is taken over all partitions $(I_k)$ of $I$ by closed intervals (non-degenerate and with disjoint interiors). 
Here $\Delta_k f := \Delta f(I_k) := f(a_k) - f(b_k)$ for $I_k=[a_k, b_k]$. 
For a partition $(I_k)$ of $I$ we denote $|(I_k)|=\max\{|I_k|\}$ and using it we define 
\[
\Volim^\phi_I(f)
:=
\limsup_{|(I_k)|\to 0} \sum_{k=1}^n \phi^+_{I_k} \bigg(\frac{|\Delta_kf|}{|I_k|}  \bigg) |I_k|
\qquad\text{and}\qquad
\Vulim^\phi_I(f)
:=
\limsup_{|(I_k)|\to 0} \sum_{k=1}^n \phi^-_{I_k} \bigg(\frac{|\Delta_kf|}{|I_k|}  \bigg) |I_k|.
\]
\end{defn}

When there is no dependence on $x$, all of these variants give the same end result (Lemma~\ref{lem:stdGrowthLimit}). 
However, with $x$-dependence, the variant with the limit superior gives a more precise result.

\begin{cor}[Variable exponent]\label{cor:preciseBVpx}
Let $f\in \BV(I)$ be left-continuous. If $p:\Omega\to [1,\infty)$ is bounded and satisfies the 
strong $\log$-Hölder condition
\[
\lim_{x\to y} |p(x)-p(y)| \log(e+\tfrac1{|x-y|}) = 0,
\]
uniformly in $y\in\Omega$, then, for $\phi(x,t):=t^{p(x)}$, we have $|D^sf|(\{p>1\})=0$ and
\[
\Volim^\phi_{I}(f) = \int_I |f'|^{p(x)}\, dx + |D^sf|(\{p=1\}). 
\]
\end{cor}

\begin{cor}[Double phase]\label{cor:preciseBVdp}
Let $f\in \BV(I)$ be left-continuous. If $a:\Omega\to [0,\infty)$ is bounded and $\alpha$-Hölder continuous 
with
\[
q < 1 + \frac\alpha n
\]
then, for $\phi(x,t):=t + a(x)t^q$, we have $|D^sf|(\{a>0\})=0$ and
\[
\Volim^\phi_{I}(f) = \int_I |f'| + a(x)\,|f'|^q\, dx + |D^sf|(\{a=0\}). 
\]
\end{cor}

We are also able to answer the question posed by Appell, Bana\'s and Merentes in \cite[p.~168]{AppBM14} 
with the help of the corollary for the Orlicz case.

\begin{cor}[Orlicz]\label{cor:orlicz} 
Let $\phi\in\Phic$ and $K:=\lim_{t\to \infty}\frac{\phi(t)}t$. Suppose that 
$V^\phi_I(f)<\infty$. 
\begin{enumerate}
\item
If $K<\infty$ and $f$ is left-continuous, then 
$f\in \BV(I)$ and 
\[
V^\phi_I(f) = \int_I \phi(|f'|)\, dx + K\, |D^sf|(I).
\]
\item
If $K=\infty$, then $f$ is absolutely continuous and $\V^\phi_I(f) = \rho_\phi(|f'|)$. 
\end{enumerate}
\end{cor}

To understand their question we recall Definition~2.11 from \cite{AppBM14}:
\[
\infty_p :=  \bigg\{\phi\in\Phic \,\bigg|\,\lim_{t\to\infty} \frac{\phi(t)}{t^p} = \infty \bigg\}.
\]
In \cite[Proposition~2.57]{AppBM14} it is shown that $\V^\phi_I$ and $\V^1_I$ generate the same space 
when $\phi\not\in \infty_1$. We use the convention that when $\phi$ is replaced by a number $p$, it indicates the 
case $\phi(x,t)=t^p$. Appell, Bana\'s and Merentes then ask whether $\V^\phi_I$ and $\V^p_I$ generate the 
same space when $\phi\not\in\infty_p$ for some $p>1$. 

However, there is not a complete analogy between the cases $p=1$ and $p>1$. Since $\phi$ is convex 
and $\phi(0)=0$, the ratio $\frac{\phi(t)}t$ is increasing so its limit always exists. Thus 
$\phi\not\in \infty_1$ is equivalent to $\displaystyle K:=\lim_{t\to \infty}\frac{\phi(t)}t < \infty$.
The same is not true when $p>1$. In fact, $\phi\not\in \infty_p$
if and only if
\[
\liminf_{t\to\infty}\frac{\phi(t)}{t^p}< \infty.
\]
However, this condition is satisfies for instance by $\phi(t)=t$ which generates the 
space $\BV(I)$ regardless of the value of $p$. 

In particular, this answers the question of Appell, Bana\'s and Merentes in the negative. 
From Corollary~\ref{cor:orlicz}(2) we 
see by \cite[Theorem~2.8.1]{DieHHR11} that $\V^\phi_I$ and $\V^p_I$ generate the same space 
if and only if 
\[
\tfrac1c t^p - c \le \phi(t)\le ct^p+c
\]
for some constant $c\ge 1$. 

\medskip

The structure of the paper is as follows. In the next section we present necessary 
background information. In Section~\ref{sect:RBV} we define functions of bounded 
Riesz $\phi$-variation when $\phi:\Omega\times[0,\infty)\to[0,\infty)$ is a generalized 
$\Phi$-function (Definition~\ref{def:variation}), and show that the variation defines a quasi-seminorm. In Section~\ref{sect:completeness} 
we show that the space of bounded Riesz $\phi$-variation is complete. Then we consider 
variants of the definition of $\phi$-variation in Section~\ref{sect:variants}, in particular the effect of 
replacing $\sup$ by $\limsup$. Finally, in Section~\ref{sect:connections} we prove a 
Riesz representation lemma, which connects the Riesz $\phi$-variation seminorm with the 
$L^\phi$-norm of the derivative of the function and prove the aforementioned corollaries.
We do this by connecting the Riesz variation in these 
cases to modern $\BV$-spaces as presented by Ambrosio, Fusco and Pallara in Section~3.2 of \cite{AmbFP00}.


\section{Preliminaries}

We briefly introduce our assumptions. More information about $L^\phi$-spaces can be found in \cite{HarH19}. 
Our previous works were based on conditions defined for almost every point $x\in \Omega$. In this article we consider 
not equivalence classes of functions but the functions themselves, and so the following assumptions have 
been correspondingly adjusted.

 We always use $I=[a, b]$ to denote a closed interval with 
end-points $a$ and $b$. 
\textit{Almost increasing} means that a function satisfies $f(s) \le L f(t)$ for all $s<t$ and some constant $L\ge 1$. 
If there exists a constant $c>0$ such that $\frac1c g(x)\le f(x) \le c g(x)$ for every $x$, then we write $f \approx g$.
Two functions $\phi$ and $\psi$ are \textit{equivalent}, 
$\phi\simeq\psi$, if there exists $L\ge 1$ such that 
$\psi(x,\frac tL)\le \phi(x, t)\le \psi(x, Lt)$ for every $x \in \Omega$ and every $t>0$.
Equivalent $\Phi$-functions give rise to the same space with comparable norms. 
By $c$ we indicate a generic positive constant whose value may change between appearances.
By $\beta$ we indicate a parameter from $(0, 1)$ which may appear in several assumptions; 
since the assumptions are all monotone in $\beta$, there is no loss of generality in assuming 
the same $\beta$.

\subsection{\texorpdfstring{$\Phi$}{Phi}-functions}

\begin{defn}
We say that $\phi: \Omega\times [0, \infty) \to [0, \infty]$ is a 
\textit{weak $\Phi$-function}, and write $\phi \in \Phiw(\Omega)$, if 
the following conditions hold:
\begin{itemize}
\item 
For every measurable function $f:\Omega\to \R$, the function $x \mapsto \phi(x, f(x))$ is measurable. 
\item
For every $x \in \Omega$, the function $t \mapsto \phi(x, t)$ is non-decreasing.
\item 
For every $x \in \Omega$, $\displaystyle \phi(x, 0) = \lim_{t \to 0^+} \phi(x,t) =0$ and $\displaystyle \lim_{t \to \infty}\phi(x,t)=\infty$.
\item 
For every $x\in \Omega$, the function $t \mapsto \frac{\phi(x, t)}t$ is $L$-almost increasing on $(0,\infty)$ with $L$ independent of $x$.
\end{itemize}
If $\phi\in\Phiw(\Omega)$ is additionally convex and left-continuous, then $\phi$ is a 
\textit{convex $\Phi$-function}, and we write $\phi \in \Phic(\Omega)$. If $\phi$ does not depend on $x$, then we omit the set and write $\phi \in \Phiw$ or $\phi \in \Phic$.
\end{defn}


We denote $\phi^+_{A}(t) := \sup_{x \in A\cap \Omega} \phi(x,t)$ and 
$\phi^-_{A}(t) := \inf_{x \in A\cap \Omega} \phi(x,t)$. 
We say that $\phi$ (or $\phi_A^\pm$) is non-degenerate if $\phi_\Omega^-, \phi_\Omega^+ \in \Phiw$; 
if $\phi$ is degenerate, then $\phi_\Omega^-|_{(0,\infty)}\equiv 0$ or $\phi_\Omega^+|_{(0,\infty)}\equiv \infty$.
We define several conditions. 
Let $p,q>0$. 
We say that $\phi:\Omega\times [0,\infty)\to [0,\infty)$ satisfies 
\begin{itemize}
\item[(A0)]\label{def:a0}
if there exists $\beta \in(0, 1]$ such that $ \phi(x, \beta) \le 1 \le \phi(x, \frac1{\beta})$ for all $x \in \Omega$; 

\item[(A1)]\label{def:a1}
if for every $K>0$ there exists $\beta \in (0,1]$ such that, for every ball $B$ and $x,y\in B \cap \Omega$,
\[ 
\phi(x,\beta t) \le \phi(y,t)+1 \quad\text{when}\quad \phi(y,t) \in \bigg[0, \frac{K}{|B|}\bigg];
\]

\item[(VA1)]\label{def:va1}
if for every $K>0$ there exists a modulus of continuity $\omega$ such that, for every ball $B=B_r$ and $x,y\in B \cap \Omega$,
\[ 
\phi(x, \tfrac t{1+\omega(r)})\le \phi(y,t) + \omega(r) \quad\text{when}\quad \phi(y,t) \in \bigg[0, \frac{K}{|B|}\bigg];
\]
%
%
%
%
%
\item[(aInc)$_p$] \label{def:aInc} if
$t \mapsto \frac{\phi(x,t)}{t^{p}}$ is $L_p$-almost 
increasing in $(0,\infty)$ for some $L_p\ge 1$ and all $x\in\Omega$;

\item[(aDec)$_q$] \label{def:aDec}
if
$t \mapsto \frac{\phi(x,t)}{t^{q}}$ is $L_q$-almost 
decreasing in $(0,\infty)$ for some $L_q\ge 1$ and all $x\in\Omega$.
\end{itemize} 
We say that \ainc{} holds if \ainc{p} holds for some $p>1$, and similarly for \adec{}.
If in the definition of \ainc{p} we have $L_p=1$, then we say that $\phi$ satisfies \inc{p}, 
similarly for \dec{q}.

\begin{eg}[Variable exponent growth]\label{ex:variable}
Let $p:\Omega\to[1,\infty]$ and let $\phi(x,t):=t^{p(x)}$ be the variable exponent functional with $t^\infty := \infty \chi_{(1,\infty)}(t)$. 
It was shown in \cite[Proposition~7.1.2]{HarH19} 
that $\phi$ satisfies \aone{} if and only if 
\[
 \bigg| \frac1{p(x)} - \frac1{p(y)} \bigg| \le \frac c {\log(e+\frac1{|x-y|})}.
\]
Note that this result does not require $p$ to be bounded. 
One can show that $\phi$ satisfies \vaone{} if
\[
 \bigg| \frac1{p(x)} - \frac1{p(y)} \bigg| \le \frac {\omega(|x-y|)}{\log(e+\frac1{|x-y|})}
\]
where $\omega$ is a function with $\lim_{r\to 0^+} \omega(r)=0$. 

It is easy to see that the variable exponent $\Phi$-function $\phi$ satisfies \inc{p^-} and \dec{p^+} provided 
$p^- \le p(x)\le p^+$ for every $x\in\Omega$ and that \azero{} always holds. 
\end{eg}

The stronger continuity condition \vaone{} (``vanishing \aone{}'') 
was introduced in \cite{HasO22, HasO_pp}. If $\phi$ satisfies 
\dec{}, then \vaone{} implies that $\phi$ is continuous. The condition also allows us to get sharper 
estimates. An example is the following Jensen's inequality with constant close to $1$ that is needed later on. 
Without \vaone{}, the inequality only holds for some $\beta>0$ which need not be 
small since $\phi^-$ is not convex \cite[Corollary~2.2.2]{HarH19}. 

\begin{thm}[Jensen's inequality]\label{thm:jensen}
If $\phi\in \Phic(\Omega)$ satisfies \vaone{} and $B=B_r\subset \Omega$, then 
there exists $L>0$ such that
\[
\phi_B^-\bigg(\frac1{1+\omega(r)} \fint_B |f|\, dx\bigg)
\le
\fint_B \phi(x, f)\, dx + \omega(r),
\]
for $\rho_\phi(Lf)\le 1$. 
\end{thm}
\begin{proof}
Let $t_0:=\|f\|_{L^1(B)}/|B|$. 
Since $\phi_B^-$ is equivalent to a convex function \cite[Lemma~2.2.1]{HarH19}, we find that 
\[
\phi_B^-(t_0) = \phi_B^-\bigg(\fint_B |f|\, dx\bigg)
\le
\fint_B \phi_B^-(L |f|)\, dx
\le 
\frac 1{|B|},
\]
where the constant $L$ is determined by the equivalence. 
This is a Jensen inequality, but the constant does not approach one for small balls. 
We obtain $\omega$ from \vaone{} and define $\beta_r:=\frac1{1+\omega(r)}$. 
We denote by $\phi'$ a function, non-decreasing in $s$, such that 
\[
\phi(x,t) = \int_0^t \phi'(x,s)\, ds.
\]
Such function exists since $\phi$ is convex in the second variable. Fix $x_0\in B$ with 
\[
\phi'(x_0, \beta_r t_0)\le (\phi')_B^-(t_0)
\]
and define 
\[
\psi(t):= \int_0^t \phi'(x_0, \min\{s, \beta_r t_0\})\, ds. 
\]
Since $\psi'$ is increasing, $\psi$ is convex. Furthermore, $\psi(t)=\phi(x_0, t)$ when $t\le \beta_r t_0$. 
It follows from Jensen's inequality that 
\[
\phi_B^-\bigg(\beta_r\fint_B |f|\, dx\bigg)
\le
\psi\bigg(\beta_r\fint_B |f|\, dx\bigg) 
\le 
\fint_B \psi(\beta_r|f|)\, dx. 
\]
When $t\le t_0$ we use \vaone{} to conclude that $\psi(\beta_r t) =\phi(x_0, \beta_r t) \le \phi(x, t)+\omega(r)$. When 
$t> t_0$ we estimate
\begin{align*}
\psi(\beta_r t) 
&= 
\psi(\beta_r t_0) + \phi'(x_0, \beta_r t_0) \beta_r(t-t_0)
\le
\phi(x, t_0)+\omega(r) + (\phi')_B^-(t_0)\beta_r(t-t_0) \\
&\le
\phi(x, t_0)+\omega(r) + \phi'(x, t_0)(t-t_0)
\le
\phi(x, t) + \omega(r),
\end{align*}
where we also used the convexity of $\phi$ in the last step. Thus 
\[
\phi_B^-\bigg(\beta_r\fint_B |f|\, dx\bigg)
\le 
\fint_B \psi(\beta_r|f|)\, dx
\le 
\fint_B \phi(x, |f|)\, dx + \omega(r). \qedhere
\]
\end{proof}

\subsection{Modular spaces}

The following results are from \cite{HarHJ_pp}; the proofs follow \cite{DieHHR11, kozlowski, Mus83}.

\begin{defn}\label{def:quasiConvexsemimodular}
Let $X$ be a real vector space. A function $\rho:X \to [0,+\infty]$ is called a \textit{quasi-semimodular} on $X$ if:
\begin{enumerate}
\item the function $\lambda\mapsto\rho(\lambda x)$ is increasing on $[0,\infty)$ for every $x\in X$;
\item $\rho(0_{X})=0$;
\item $\rho (-x)=\rho(x)$ for every $x\in X$;
\item there exists $\beta\in (0,1]$ such that $\rho(\beta(\theta x +(1-\theta)y) ) \le \theta\rho(x) + (1-\theta)\rho(y)$ for every $x,y\in X$ and every $\theta \in [0,1] $.
\end{enumerate}
If (4) holds with $\beta=1$, then $\rho$ is a \textit{semimodular}.

If $\rho$ is a quasi-semimodular in $X$, then the set defined by 
\[
X_\rho
:=
\big\{ x \in X \,\big|\, \lim_{\lambda \to 0} \rho(\lambda x)=0 \big\}
\]
is called a \textit{modular space}. We define the \textit{Luxemburg quasi-seminorm} on $X_\rho$ by
\[
\|x\|_\rho 
:= 
\inf \Big\{ \lambda>0 \, \Big|\,\rho \Big( \frac{x}{\lambda}  \Big) \leqslant 1 \Big\}.
\]
\end{defn}

Note that our terminology differs from Musielak \cite{Mus83}. Our justification comes from 
the correspondence with standard terminology for norms, as demonstrated in the following proposition. 

\begin{prop}\label{prop:LuxemburgNorm}
Let $X$ be a real vector space.
\begin{enumerate}
\item If $\rho$ is a quasi-semimodular in $X$, then $\|\cdot\|_\rho$ is a quasi-seminorm.
\item If $\rho$ is a semimodular in $X$, then $\|\cdot\|_\rho$ is a seminorm.
\item If $\rho$ is a quasi-modular in $X$, then $\|\cdot\|_\rho$ is a quasi-norm.
\item If $\rho$ is a modular in $X$, then $\|\cdot\|_\rho$ is a norm.
\end{enumerate}
\end{prop}

The next proposition contains the main properties that we need regarding modular spaces. 

\begin{lem}[\cite{HarHJ_pp}]\label{lem:modularInequalities}
Let $X$ be a real vector space, $\rho$ be a quasi-semimodular on $X$ and $x \in X$.
Denote by $\beta$ the constant in property (4) of Definition \ref{def:quasiConvexsemimodular}.
Then
\begin{enumerate}
\item $\|x\|_\rho<1\implies \rho(x)\leq 1\implies \|x\|_\rho\leqslant 1$;
\item $\|x\|_\rho < 1 \implies \beta\rho(x)\leqslant \|x\|_\rho$;
\item $\|x\|_\rho > 1 \implies\rho(x)\geqslant\beta\|x\|_\rho$;
\item $\|x\|_\rho \leqslant \beta^{-1}\rho(x)+1$.
\end{enumerate}
\end{lem}

As special cases we have generalized Orlicz and Orlicz--Sobolev spaces.

\begin{defn}\label{def:Lphi}
Let $\phi \in \Phiw(\Omega)$ and define the quasi-semimodular 
$\varrho_\phi$ for $f\in L^0(\Omega)$, the set of measurable functions in $\Omega$, by 
\begin{align*}
\varrho_\phi(f) &:= \int_\Omega \phi(x, |f(x)|)\,dx.
\end{align*}
The \emph{generalized Orlicz space}, also called Musielak--Orlicz space, is defined as the set 
\begin{align*}
L^\phi(\Omega) := (L^0(\Omega))_{\rho_\phi} = 
\big\{f \in L^0(\Omega) \,\big|\, \lim_{\lambda \to 0^+} \varrho_\phi(\lambda f) = 0\big\}
\end{align*}
equipped with the (Luxemburg) quasi-seminorm 
\begin{align*}
\|f\|_{L^\phi(\Omega)}:= \|f\|_{\rho_\phi} 
= \inf  \bigg\{ \lambda>0 \, \bigg|\, \varrho_\phi \Big(\frac{f}{\lambda}  \Big) \leq 1 \bigg\}.
\end{align*}
If the set is clear from context we abbreviate $\|f\|_{L^\phi(\Omega)}$ by $\|f\|_{\phi}$.
\end{defn}


\begin{defn}
A function $f \in L^\phi(\Omega)$ belongs to the
\textit{Orlicz--Sobolev space $W^{1, \phi}(\Omega)$} if its weak derivative $f'$ exists and belongs to $L^{\phi}(\Omega)$.
For $f\in W^{1,\phi}(\Omega)$, we define the norm
\[
\| f \|_{W^{1,\phi}(\Omega)} := \| f \|_{L^\phi(\Omega)} + \| f' \|_{L^\phi(\Omega)}.
\]
\end{defn}

\section{Bounded Variation in the Riesz Sense}\label{sect:RBV}

We introduce the space of bounded Riesz $\phi$-variation based on $\V^\phi_I$ from the introduction. Note that we do not 
assume that $f$ is measurable.

\begin{defn} \label{def:RBV}
Let $I\subset \R$ be a closed interval and $\phi\in \Phiw(I)$.
The space of \emph{bounded $\phi$-variation in Riesz' sense} is defined by the 
quasi-semimodular $\V^\phi_{I}$: 
\[
\RBV^\phi(I):=\big\{ f : I\to \R \, \big|\, \lim_{\lambda\to 0}\V^\phi_I(\lambda f)=0\big\}.
\]
\end{defn}

Often two $\Phi$-functions $\phi, \psi\in\Phiw(I)$ are considered to be the same, if $\phi(x,t)=\psi(x,t)$ for almost everywhere $x$ and every $t$.
In our setting we cannot use this convention, as the following example demonstrates.

\begin{eg}
Let $(A_j)_{j=1}^\infty$ be a sequence of pairwise disjoint, countable and dense subsets of $I$.
Denote $A:=\bigcup_{j=1}^\infty A_j$.
Then $|A|=0$ since $A$ is countable.
Define $\phi, \psi:I\times[0,\infty)\to[0,\infty]$ by $\phi(x,t):=t^2$ and 
\[
\psi(x,t):=\begin{cases}
t^2, & \text{if }x\in I\setminus A, \\
jt^2, & \text{if }x\in A_j\text{ for some }j\in\N.
\end{cases}
\]
Then $\psi(x,t)=\phi(x,t)$ for every $x\in I\setminus A$ and every $t\geq 0$, so that $\phi=\psi$ a.e.
Let $f:I\to\R$ be the identity function and let $(I_k)$ be any partition of $I$.
Then $|\Delta_kf|/|I_k|=1$ for every $k$.
Since the sets $A_j$ are dense, the intersection $A_j\cap I_k$ is non-empty for every $j$ and every $k$.
Thus $\psi_{I_k}^+(1)\geq j$ for every $j$, and therefore
\[
\V_I^\psi(f)
\geq
\sum_{k=1}^n \psi^{+}_{I_k}  \bigg(\frac{|\Delta_kf|}{|I_k|}  \bigg)|I_k|
\ge
\sum_{k=1}^n j|I_k|
=
j|I|.
\]
Letting $j\to\infty$, we see that $\V^\psi_{I}(f)=\infty$.
On the other hand, 
\[
\sum_{k=1}^n \phi^{+}_{I_k}  \bigg(\frac{|\Delta_kf|}{|I_k|}  \bigg)|I_k|
= \sum_{k=1}^n \phi^{+}_{I_k} (1)|I_k|
= \sum_{k=1}^n |I_k|
= |I|.
\]
This implies that $\V^\phi_{I}(f)=|I|$.
\end{eg}

Next, we show that equivalent $\Phi$-functions give rise to the same space of bounded variation.

\begin{lem}\label{lem:equivalence}
Suppose that $\phi,\psi\in\Phiw(I)$ and $\phi\simeq\psi$ with constant $L\ge 1$.
Then 
\[
\V^\psi_{I}(L^{-1}f) \leq \V^\phi_{I}(f) \leq \V^\psi_{I}(Lf)
\]
for every $f:I\to\R$ and $\RBV^\phi(I)=\RBV^\psi(I)$.
\end{lem}

\begin{proof}
By the definition of equivalence, 
\[
\psi(x,L^{-1}t)\le \phi(x, t)\le \psi(x, Lt)
\]
for every $x \in I$ and every $t\in[0,\infty)$.
Thus
\[
\V^\phi_{I}(f)
=\sup_{(I_k)} \sum_{k=1}^n \phi^+_{I_k} \bigg(\frac{|\Delta_kf|}{|I_k|}  \bigg) |I_k|
\leq \sup_{(I_k)} \sum_{k=1}^n \psi^+_{I_k} \bigg(L\frac{|\Delta_kf|}{|I_k|}  \bigg) |I_k|
=
\V^\psi_{I}(Lf).
\]
Similarly $\V^\psi_{I}(L^{-1}f) \leq \V^\phi_{I}(f)$.
To see that $\RBV^\phi(I)=\RBV^\psi(I)$, we note that if $\V^\phi_{I}(\lambda f)\to 0$, then $ \V^\psi_{I}(L^{-1}\lambda f)\to 0$, and if $\V^\psi_{I}(\lambda f)\to 0$, then $ \V^\phi_{I}(L^{-1}\lambda f)\to 0$.
\end{proof}



We now show that $\V_I^\phi$ is a quasi-semimodular.

\begin{lem} \label{lem:modular}
Let $\phi\in\Phiw(I)$. Then $\V^\phi_{I} : \mathbb R^{I} \to \overline{\mathbb R}_+ $ is a quasi-semimodular.
If $\phi\in\Phic(I)$, then $\V_I^\phi$ is a semimodular.
\end{lem}

\begin{proof}
The properties $\V^\phi_{I}(0)=0$ and $\V^\phi_{I}(-f)= \V^\phi_{I}(f)$ are clear.
Since $t\mapsto\phi(x,t)$ is increasing for every $x\in I$ it follows that $t\mapsto\phi_A^+(t)$ is increasing whenever $A\subset I$.
Since\[
\V^\phi_{I}(\lambda f)
=\sup_{(I_k)} \sum_{k=1}^n \phi^+_{I_k} \bigg(\lambda\frac{|\Delta_kf|}{|I_k|}  \bigg) |I_k|,
\]
the function $\lambda\mapsto\V^{\phi}_{I}(\lambda f)$ is increasing on $[0,\infty)$ for every $f:I\to\R$.

Note that $\phi_A^+$ satisfies \ainc{1} with the same constant $L\geq 1$ as $\phi$.
Let $(I_k)$ be a partition of $I$.
If $\theta\in[0,1]$, then
\[
\begin{aligned}
\phi_{I_k}^+ \bigg(\frac{1}{2L}\frac{|\Delta_k (\theta f + (1-\theta)g)|}{|I_k|}  \bigg) 
& \leq
\phi_{I_k}^+ \bigg(\frac{\theta}{2L}\frac{|\Delta_kf|}{|I_k|}+\frac{1-\theta}{2L}\frac{|\Delta_k g|}{|I_k|}  \bigg) \\
& \leq 
\phi_{I_k}^+ \bigg(\frac\theta L\frac{|\Delta_k f|}{|I_k|}  \bigg)
+\phi_{I_k}^+ \bigg(\frac{1-\theta}L\frac{|\Delta_k g|}{|I_k|}  \bigg) \\
& \leq
\theta\phi_{I_k}^+ \bigg(\frac{|\Delta_k f|}{|I_k|}  \bigg)
+(1-\theta)\phi_{I_k}^+ \bigg(\frac{|\Delta_k g|}{|I_k|}  \bigg),
\end{aligned}
\]
where the last inequality follows from \ainc{1}.
This implies that $\V^\phi_{I}(\beta(\theta f +(1-\theta)g)) \leq \theta\V^\phi_{I}(f) + (1-\theta)\V^\phi_{I}(g)$ with $\beta:=\frac 1{2L}$.

If $\phi\in\Phic(I)$, then $\phi_{I_k}^+$ is convex as the supremum of convex functions.
Thus
\[
\begin{aligned}
\phi_{I_k}^+ \bigg(\frac{|\Delta_k (\theta f + (1-\theta)g)|}{|I_k|}  \bigg) 
& \leq
\phi_{I_k}^+ \bigg(\theta\frac{|\Delta_kf|}{|I_k|}+(1-\theta)\frac{|\Delta_k g|}{|I_k|}  \bigg) \\
& \leq 
\theta\phi_{I_k}^+ \bigg(\frac{|\Delta_k f|}{|I_k|}  \bigg)+(1-\theta)\phi_{I_k}^+ \bigg(\frac{|\Delta_k g|}{|I_k|}  \bigg),
\end{aligned}
\]
and it follows that $\V^\phi_{I}$ is convex.
\end{proof}

By Proposition~\ref{prop:LuxemburgNorm}, we can define the Luxemburg quasi-seminorm in $\RBV^\phi(I)$ by
\[
\|f\|_{\RBV^\phi(I)}= \inf \bigg\{\lambda >0 \, \Big|\, \V^\phi_{I}  \Big( \frac{f}{\lambda}  \Big) 
\le 1 \bigg\}.
\]
This is not a quasi-norm, as is easily seen by considering constant functions: if $f$ is constant, then $\V^\phi_{I}(f/\lambda)=0$ for every $\lambda>0$ (since $\Delta_kf=0$ for every partition and every $k$), and therefore $\|f\|_{\RBV^\phi(I)}=0$.
It is also easy to see that $\V^\phi_{I}(f)=\V^\phi_{I}(g)$ whenever $f-g$ is constant.
Thus, in addition to $\RBV^\phi(I)$, we also consider the sub-space 
\[
\RBVz^\phi([a, b]) := \{f\in\RBV^\phi([a, b])\mid f(a)=0\}.
\]
Then $\RBVz^\phi(I)$ is a quasi-normed space with the quasi-norm $\|\cdot\|_{\RBV^\phi(I)}$, as we will see in the next theorem.

\begin{thm}
Let $\phi\in\Phiw(I)$.
Then $\RBVz^\phi(I)$ is a quasi-normed space which is non-trivial if and only if $\phi^+_{I}$ is non-degenerate.
\end{thm}
\begin{proof}
By Proposition~\ref{prop:LuxemburgNorm} and Lemma~\ref{lem:modular}, $\RBVz^\phi(I)$ is a quasi-seminormed space.
For $\RBVz^\phi(I)$ to be a quasi-normed space, we check that $\|f\|_{\RBVz^\phi(I)}=0$ only if $f=0$.
This is equivalent to the condition that $\|f\|_{\RBV^\phi(I)}$ is non-zero for non-constant $f$
since $f(a)=0$.

Let $f\in\RBV^\phi(I)$ be an arbitrary non-constant function.
Then we can choose a partition $(I_k)$ of $I$ such that $\Delta_{k_0}f\neq 0$ for some $k_0$.
Note that $\lim_{t\to\infty}\phi_{I_{k_0}}^+(t)=\infty$ by the definition of $\Phiw(I)$.
Hence, by the definition of $\V_{I}^\phi$, we get that
\begin{equation*}
\lim_{\lambda\to 0^+}\V_{I}^\phi \Big(\frac{f}{\lambda} \Big)
\geq
\lim_{\lambda\to 0^+} \sum_{k=1}^n \phi_{I_{k}}^+ \bigg(\frac{|\Delta_{k}f|}{\lambda|I_k|}  \bigg)|I_k|
\geq
\lim_{\lambda\to 0^+} \phi_{I_{k_0}}^+ \bigg(\frac{|\Delta_{k_0}f|}{\lambda|I_{k_0}|}  \bigg)|I_{k_0}|
=\infty.
\end{equation*}
Thus $\V_{I}^\phi(f/\lambda)>1$ when $\lambda$ is small enough, and therefore $\|f\|_{\RBV^\phi(I)}>0$.

We now prove the claim concerning the non-triviality of $\RBVz^\phi(I)$.
Suppose first that $\phi_{I}^+$ is degenerate.
Let $f$ be a non-constant function.
Let us show that $\V_I^\phi(f)=\infty$.
If $f(a)\neq f(b)$, then $|\Delta f(I)|/|I|>0$, and
\[
\V_{I}^\phi(f)
\geq
\phi_{I}^+ \bigg(\frac{|\Delta f(I)|}{|I|} \bigg)|I|
=\infty.
\]
If $f(a)=f(b)$, then there exists $c\in I$ with $f(c)\neq f(a)$, since $f$ is not constant.
Let $I_1:=[a,c]$ and $I_2:=[c,b]$.
Since $\phi_{I}^+$ is degenerate, it follows that at least one of $\phi_{I_1}^+$ or $\phi_{I_2}^+$ must also be degenerate.
Since $|\Delta_k f|/|I_k|>0$, $k=1,2$, we get that
\[
\V_{I}^\phi(f)
\geq
\sum_{k=1}^2\phi_{I_k}^+ \bigg(\frac{|\Delta_k f|}{|I_k|}  \bigg)|I_k|
=\infty.
\]
Thus $\V_I^\phi(\lambda f)=\infty$ for every non-constant function $f$ and every
$\lambda\in(0,\infty)$.
Hence $\RBV^\phi(I)$ is just the space of constant functions, which implies that $\RBVz^\phi(I)$ is trivial. 

Suppose then that $\phi^+_I$ is non-degenerate.
Thus there exists $t_0\in(0,\infty)$ with $\phi^+_{I}(t_0)<\infty$.
Let $g(x):=t_0(x-a)$. If $(I_k)$ is any partition of $I$, then
\[
\sum_{k=1}^n \phi_{I_{k}}^+ \bigg(\frac{|\Delta_{k}g|}{|I_k|}  \bigg)|I_k|
=
\sum_{k=1}^n \phi_{I_{k}}^+(t_0)|I_k|
\leq
\sum_{k=1}^n \phi_{I}^+(t_0)|I_k|
=\phi^+_I(t_0)|I|
<
\infty,
\]
which, together with \ainc{1}, implies that $g\in\RBV^\phi(I)$.
Since $g(a)=0$, it follows that $g\in\RBVz^\phi(I)$.
Thus $\RBVz^\phi(I)$ is non-trivial.
\end{proof}


\section{Completeness}\label{sect:completeness}

In this section, we show the completeness of the space of bounded Riesz $\varphi$-variation.

\begin{lem}\label{lem:modularnormab}
Let $\varphi \in \Phiw(I)$ and $L$ be the constant from \ainc{1}. Then, for $\alpha\ge 0$ and $\beta>0$,
\[
 \V^\phi_{I}  \Big(\frac{f}{\beta}  \Big) 
 \le \alpha \implies \|f\|_{\RBV^\phi(I)}
 \le
 \begin{cases}
 L \alpha \beta, & \alpha >1,\\
 \beta, & \alpha \leqslant 1.
 \end{cases}
\]
\end{lem}
\begin{proof} 
The case $\alpha \leqslant 1$ follows immediately from the definition of norm. 

Let now $\alpha >1$. Since $\varphi^+_{I_k}$ satisfies \ainc{1} and $\varphi^+_{I_k}(0)=0$, we have 
\[
\varphi^+_{I_k}  \bigg( \frac{|\Delta_k f|}{L\alpha \beta |I_k| }  \bigg) 
\leqslant 
\frac{1}{\alpha} \varphi^+_{I_k}  \bigg( \frac{|\Delta_k f|}{ \beta |I_k|}  \bigg)
\] for each sub-interval of $I$, from which we obtain 
\[
\V^\phi_{I}  \Big(\frac{f}{L \alpha \beta}  \Big) 
\le
\frac{1}{\alpha} \V^\phi_{I}  \Big(\frac{f}{\beta}  \Big) 
\le
 1.
\]
The result now follows from the definition of the norm.
\end{proof}

\begin{lem}\label{lem:embedding}
Let $\phi, \psi \in \Phiw(I)$. Suppose that there exists $K>0$ such that
\begin{equation*}
\psi  \Big(x, \frac{t}{K} \Big) \le \phi(x, t)+1.
\end{equation*} 
Then
\begin{equation*}
\RBV^{\phi}(I) \hookrightarrow \RBV^\psi(I).
\end{equation*}
\end{lem}

\begin{proof} 
Let $\lambda_\epsilon:=\|f\|_{\RBV^\phi(I)} + \epsilon$. Then $\V_{I}^{\phi} (f/\lambda_\epsilon) \le 1$.
From the assumption and the definition of the modular, we have
$\V_{I}^{\psi} ( f/(\lambda_\epsilon K) ) \le \V_{I}^{\phi}(f/\lambda_\epsilon) + |I| \le 1+|I|$.
By Lemma~\ref{lem:modularnormab}, 
\[
\|f\|_{\RBV^\psi (I)} \leq L K (1+|I|) \lambda_\varepsilon = L K (1+|I|) (\|f\|_{\RBV^\phi(I)} + \varepsilon),
\]
which concludes the proof.
\end{proof}

We now show that $\RBV^\phi(I)$ contains all Lipschitz functions and is contained in the set 
of absolutely continuous functions when $\phi$ satisfies suitable conditions.
We denote by $\Lip(I)$ the space of Lipschitz functions on $I$ and by $\AC(I)$ the space of absolutely continuous functions in $I$.

\begin{lem}\label{lem:embedding2}
Let $\phi\in \Phiw(I)$ satisfy \azero{}, \ainc{} and \adec{}. Then 
\[
\Lip(I)\subseteq \RBV^\phi(I) \subseteq \AC(I).
\]
\end{lem}
\begin{proof}
From \cite[Proposition~2.52]{AppBM14}, we know that 
\[
\Lip(I)\subseteq \RBV^s(I) \subseteq \AC(I), 
\]
with $ s \in (1, \infty)$, where $\RBV^s(I)$ corresponds to the choice $\phi(x,t)=t^s$.
When $t\ge \frac1\beta$, it follows from \ainc{p} that 
$(\beta t)^p\phi(x,\frac1\beta)\le L\phi(x,t)$ and so it follows from \azero{} that 
\[
(\beta t)^p\le L\phi(x,t) + 1.
\]
From \adec{q} we conclude that $\phi(x,\beta t)\le Lt^q \phi(x,\beta)$ when $t\ge 1$. Hence \azero{} implies that 
\[
\phi(x,\beta t)\le Lt^q + 1.
\]
By Lemma~\ref{lem:embedding} and these inequalities, we conclude that
\[
\Lip(I)\subseteq \RBV^q(I) \subseteq \RBV^\phi (I)\subseteq \RBV^p(I)\subseteq \AC(I). \qedhere
\]
\end{proof}

\begin{thm}
If $\phi \in \Phiw(I)$, then $\RBVz^\phi(I)$ is a quasi-Banach space.
\label{theo:B-space}
\end{thm}
\begin{proof}
Fix $I=[a,b]$. We first prove that 
\begin{equation}\label{eq:Linftyembedding}
\sup_I |f| \leq C \|f\|_{\RBV^{\phi}(I)}
\end{equation}
when $f\in \RBVz^{\phi}(I)$. 
Denote $I_1=[a, x]$. Since $f(a)=0$, it follows from $\Delta f(I_1)=f(x)$ that 
\[
\phi  \bigg(a, \frac{|f(x)|}{|I_1|}  \bigg) |I_1| 
\le 
\phi_{I_1}^+  \bigg( \frac{|\Delta_{I_1} f|}{|I_1|}  \bigg) |I_1| \leq \V^{\phi}_I(f). 
\]
We apply this inequality to the function $\frac{f}{\|f\|_{\RBV^{\phi}(I)}+\epsilon}$ 
where $\epsilon>0$. Thus
\[
\phi\bigg(a, \frac{|f(x)|}{|I_1|(\|f\|_{\RBV^{\phi}(I)}+\epsilon)} \bigg) 
\le 
\V^{\phi}_I\bigg(\frac{f}{\|f\|_{\RBV^{\phi}(I)}+\epsilon} \bigg) 
\le
1.
\]
Since $\phi(a,t)\to \infty$ as $t\to \infty$, we conclude that the argument of $\phi$ on the 
left-hand side is bounded. Thus 
\begin{equation*}\label{eq:pointwise}
|f(x)| \le C\,|I|\, (\|f\|_{\RBV^{\phi}(I)} + \epsilon).
\end{equation*}
 The estimate \eqref{eq:Linftyembedding} now follows as $\epsilon \to 0^+$.
 
Let $(f_i)$ be a Cauchy sequence in $\RBV^\phi_0(I)$. 
By \eqref{eq:Linftyembedding}, it is also a Cauchy sequence in $L^\infty(I)$, which ensures that $f:= \lim _{i \to \infty} f_i$ exists. 
For $\epsilon >0$, there exists $N_\epsilon \in \mathbb N$ such that $i,j > N_\epsilon$ implies 
$\V ^{\phi}_I(\frac{f_i -f_j}\epsilon) \leq 1$, since $\|f_i-f_j\|_{\RBV^{\phi}(I)} < \epsilon$. 
By \cite[Lemma~2.2.1]{HarH19} there exists $\psi\in \Phic(I)$ with $\psi\simeq\phi$. 
Since $\psi$ is left-continuous, $\psi^+_{I_k}$ is left-continuous (see comment after the proof of \cite[Lemma~2.1.8]{HarH19}) and lower semicontinuous by \cite[Lemma~2.1.5]{HarH19}. Taking this into account, we get
\[
\V^{\psi} \Big( \frac{f_i-f}{L\epsilon}, (I_k) \Big) 
= 
\V^{\psi} \Big( \frac{f_i- \lim_{j \to \infty}f_j}{L\epsilon}, (I_k) \Big) 
\le 
\liminf_{j\to \infty} \V ^{\psi} \Big( \frac{f_i- f_j}{L\epsilon}, (I_k) \Big). 
\]
When $i > N_\epsilon$, it follows from Lemma~\ref{lem:equivalence} that 
\begin{align*}
\V^{\phi} \Big( \frac{f_i-f}{L^2\epsilon}, (I_k) \Big) 
&\le
\V^{\psi} \Big( \frac{f_i-f}{L\epsilon}, (I_k) \Big) 
\le
\liminf_{j\to \infty} \V^{\psi} \Big( \frac{f_i-f_j}{L\epsilon}, (I_k) \Big) \\
&\le
\liminf_{j\to \infty} \V^{\phi} \Big( \frac{f_i-f_j}{\epsilon}, (I_k) \Big)
\le
\liminf_{j\to \infty} \V_I^{\phi} \Big( \frac{f_i-f_j}{\epsilon}\Big) 
\le 
1.
\end{align*}
Taking the supremum over partitions $(I_k)$, we find that 
\[
\V_I^{\phi} \Big( \frac{f_i-f}{L^2\epsilon}\Big) \le 1
\qquad\text{and so}\qquad
\|f_i-f\|_{\RBV^{\phi}} \le L^2\epsilon.
\]
Thus $f \in \RBV ^{\phi}_0(I)$ and $(f_i)$ converges to $f$ in $\RBV^{\phi}_0(I)$.
\end{proof}


\section{Variants of the variation}\label{sect:variants}

In Definition~\ref{def:variation} we defined $\V^\phi_I$, $\Volim^\phi_I$ and $\Vulim^\phi_I$ 
by taking supremum or limit superior of partitions and using either $\phi^+$ or $\phi^-$. 
In this section, we consider the impact of these choices. 
Since all the partitions in $\Volim^\phi_I$ are allowed in the supremum in $\V^\phi_I$ and $\phi^-_{I_k}\le\phi^+_{I_k}$, 
we see that 
\[
\Vulim^\phi_I\le \Volim^\phi_I \le \V^\phi_I.
\]
When $\phi$ is convex and independent of $x$, the 
opposite inequalities hold, as well. Indeed, $\Volim^\phi_{I}(f)=\Vulim^\phi_{I}(f)$ 
is trivial in this case.

\begin{lem}\label{lem:stdGrowthLimit}
If $\phi \in\Phic$, then
\[
\V^\phi_{I}(f)
=
\Volim^\phi_{I}(f).
\]
\end{lem}
\begin{proof}
Let $(I_k)$ be a partition of $I$ and $(I_k')$ be its subpartition. 
Suppose $I_k=\cup_{j=j_k}^{j_{k+1}-1} I_j'$. Then 
\[
\frac{|\Delta f(I_k)|}{|I_k|} 
\le
\sum_{j=j_k}^{j_{k+1}-1}\frac{|\Delta f(I_j')|}{|I_k|} 
=
\sum_{j=j_k}^{j_{k+1}-1}\frac{|I_j'|}{|I_k|} \frac{|\Delta f(I_j')|}{|I_j'|}. 
\]
Since the coefficients' sum equals 1, it follows from convexity that 
\[
\phi \bigg(\frac{|\Delta f(I_k)|}{|I_k|} \bigg) |I_k|
\le
\sum_{j=j_k}^{j_{k+1}-1}\frac{|I_j'|}{|I_k|} \phi \bigg(\frac{|\Delta f(I_j')|}{|I_j'|} \bigg)|I_k|
=
\sum_{j=j_k}^{j_{k+1}-1} \phi \bigg(\frac{|\Delta f(I_j')|}{|I_j'|} \bigg)|I_j'|.
\]
Thus $\V^\phi(f, (I_k)) \le \V^\phi(f, (I_k'))$. Since this holds for any subpartition, we 
see that we can always move to a partition with half as large   $|(I_k)|$ but no smaller 
$\V^\phi(f, (I_k))$. Hence $\V^\phi_I\le \Volim^\phi_I$.
\end{proof}

We just showed that $\Volim^\phi=V^\phi$ when $\phi\in\Phic$. 
However, if $\phi$  depends on $x$, then it is possible that $\Volim^\phi_I < \V^\phi_I$.

\begin{eg}\label{eg:limitDifferent}
Let $I:=[0,1]$ and define $\phi\in\Phic(I)$ by $\phi(x,t):=(x+1)t^2$.
Note that $\phi$ satisfies \azero, \aone, \ainc2 and \adec2.
Let $f$ be the identity function.
Then
\[
\V_I^\phi(f)
\geq
\phi^+_I \bigg(\frac{|\Delta f(I)|}{|I|}  \bigg) |I|
=\phi^+_I(1)
=2.
\]
Let next $\epsilon\in(0,1)$ and let $(I_k)$ be a partition with $|(I_k)|<\epsilon$. 
If $I_k=[a_k, b_k]$, then 
\[
\phi \bigg(a_k, \frac{|\Delta_k f|}{|I_k|}  \bigg)
\le
\phi^+_{I_k} \bigg(\frac{|\Delta_k f|}{|I_k|}  \bigg)
\le
\phi \bigg(b_k, \frac{|\Delta_k f|}{|I_k|}  \bigg)
\le
\phi \bigg(a_k+\epsilon, \frac{|\Delta_k f|}{|I_k|}  \bigg)
\]
since $\phi$ is increasing in $x$. If $f$ is the identity function, then $\frac{|\Delta_k f|}{|I_k|}=1$ and 
\[
\int_0^1 \phi(x,1)\, dx
\le
\sum_{k=1}^{n}\phi^+_{I_k} \bigg(\frac{|\Delta_k f|}{|I_k|}  \bigg) |I_k|
\le
\int_\epsilon^{1+\epsilon} \phi(x,1)\, dx.
\]
Since $\phi(x,1)=x+1$, we obtain as $\epsilon\to 0$ that
\[
\Volim_I^\phi(f) = \int_0^1 \phi(x,1)\, dx = \frac32.
\]
\end{eg}

We next show that $\phi^+$ and $\phi^-$ give the same result at the limit under 
the stronger continuity condition \vaone{}. Example~\ref{eg:vaNeeded} shows that 
the result does not hold under \aone{}.

\begin{lem}\label{lem:phiMinusLimit}
If $\phi \in\Phic(\Omega)$ satisfies \vaone{} and \dec{}, then
\[
\Volim^\phi_I(f)
=
\underline{\mathrm V}^\phi_I(f).
\]
\end{lem}
\begin{proof}
Clearly, $\Vulim^\phi_I\le \Volim^\phi_I$ so we show that $\Volim^\phi_I\le\Vulim^\phi_I$. 
We assume that $\Vulim^\phi_I(f)<\infty$, since otherwise there is nothing to prove. 
For $\epsilon\in (0,1)$ we choose $\delta\in (0, \epsilon)$ such that 
\[
\sup_{|(I_k)|<\delta} \sum_{k=1}^n \phi^-_{I_k} \bigg(\frac{|\Delta_kf|}{|I_k|}  \bigg) |I_k| \le (1+\epsilon) \Vulim^\phi_I(f).
\]
Let $(I_k)$ be a partition with $|(I_k)|<\delta$. By the previous line,  
\[
\phi^-_{I_k} \bigg(\frac{|\Delta_kf|}{|I_k|}  \bigg) < \frac{2 \Vulim^\phi_I(f)}{|I_k|}.
\]
Therefore it follows by \vaone{} and \dec{q} that 
\[
\phi^+_{I_k} \bigg(\frac{|\Delta_kf|}{|I_k|}  \bigg)
\le 
(1+\omega(\epsilon))^q\bigg[\phi^-_{I_k} \bigg(\frac{|\Delta_kf|}{|I_k|}  \bigg) + \omega(\epsilon)\bigg],
\]
where $\omega$ is a modulus of continuity from \vaone{} with $K:=2 \Vulim^\phi_I(f)$.
Hence
\begin{align*}
\sum_{k=1}^n \phi^+_{I_k} \bigg(\frac{|\Delta_kf|}{|I_k|}  \bigg) |I_k| 
& \le
(1+\omega(\epsilon))^q \bigg[\sum_{k=1}^n \phi^-_{I_k} \bigg(\frac{|\Delta_kf|}{|I_k|}  \bigg) |I_k| +\omega(\epsilon) |I|\bigg]\\
& \le
(1+\omega(\epsilon))^q\big[(1+\epsilon)\Vulim^\phi_I(f)+\omega(\epsilon) |I|\big].
\end{align*}
Since this holds for all $(I_k)$ with $|(I_k)|<\delta$, we conclude that 
\[
\Volim^\phi_I(f)\le (1+\omega(\epsilon))^q[(1+\epsilon)\Vulim^\phi_I(f)+\omega(\epsilon) |I|]
\to \Vulim^\phi_I(f),
\]
as $\epsilon\to 0$. Combined with $\Vulim^\phi_I\le \Volim^\phi_I$, this gives 
$\Vulim^\phi_I= \Volim^\phi_I$.
\end{proof}

\begin{eg}\label{eg:vaNeeded}
Let $p(x):=1+\frac1{\log(1/x)}$ for $x\in(0,\frac12]$ and $p(0):=1$. Set $\phi(x,t):=t^{p(x)}$. We consider the function 
$f:=\chi_{(0,\frac12]}$ and a partition $(I_k)$ of $I:=[0,\frac12]$. Then $\Delta f(I_k)=0$ unless $k=1$. Suppose that 
$I_1:=[0, x]$, $x\in (0, \frac12)$. Then 
\[
\phi^+_{I_1}\bigg(\frac{|\Delta f(I_1)|}{|I_1|}\bigg)|I_1| = x^{1-p^+_{I_1}} = x^{-\frac1{\log(1/x)}} = e
\]
and
\[
\phi^-_{I_1}\bigg(\frac{|\Delta f(I_1)|}{|I_1|}\bigg)|I_1| = x^{1-p^-_{I_1}} = x^{0} = 1.
\]
Since all other terms vanish, we see that $\Volim^\phi_I(f)=e$ and $\Vulim^\phi_I(f)=1$. 
Note that $p$ is $\log$-Hölder continuous so that $\phi$ satisfies \aone{} \cite[Proposition~7.1.2]{HarH19}.
\end{eg}

The next example shows that $\phi^+$ and $\phi^-$ do not give the same result without 
the limiting process, i.e.\ for $V^\phi_I$ and a version of it with $\phi^-_{I_k}$.

\begin{eg}
Let $I:=[0,3]$ and define $\phi\in\Phic(I)$ by
\[
\phi(x,t) := t^{\max\{2,x\}}.
\]
Note that $\phi$ satisfies \azero{}, \aone{}, \vaone{}, \ainc2 and \adec3.
For $\alpha > \frac3\beta$ we define $f_\alpha:I \to \R$ by
$f_\alpha(x) := \alpha\min\{x, 1\}$. Then
\[
\V_I^\phi(\beta f_\alpha)
\geq
\phi^+_I \bigg(\frac{\beta|\Delta f_\alpha(I)|}{|I|}  \bigg) |I|
= 3 \phi^+_I \Big(\frac{\beta \alpha}{3} \Big)
= \frac{(\beta \alpha)^3}{9}.
\]
Let $(I_k)$ be a partition of $I$. If $I_k \cap [0,1]$ is non-empty, then
\[
\phi_{I_k}^- \bigg(\frac{|\Delta_k f_\alpha|}{|I_k|}  \bigg)
\leq 
\phi_{I_k}^-(\alpha)
\le
\alpha^2,
\]
since $|\Delta_k f_\alpha| \leq \alpha\,|I_k|$. If $I_k \cap [0,1]$ is empty, then $\Delta_k f_\alpha = 0$.
Thus
\[
\sum_{k=1}^{n}\phi^-_{I_k} \bigg(\frac{|\Delta_k f_\alpha|}{|I_k|}  \bigg) |I_k|
\leq \sum_{k=1}^{n} \alpha^2|I_k|
= 3\alpha^2,
\]
and hence
\[
\sup_{(I_k)} \sum_{k=1}^n \phi^-_{I_k} \bigg(\frac{|\Delta_kf_\alpha|}{|I_k|}  \bigg) |I_k|
\leq 
3\alpha^2
\le 
\frac{27}{\beta^3 \alpha}\V_I^\phi(\beta f_\alpha).
\]
As $\alpha\to\infty$, this shows that the inequality $\V_I^\phi(\beta f_\alpha) \le 
\sup_{(I_k)} \sum_{k=1}^n \phi^-_{I_k}(\frac{|\Delta_kf_\alpha|}{|I_k|}) |I_k|$ 
does not hold for any fixed $\beta>0$.
\end{eg}

The next lemma shows that $V^\phi_I$ and $\Volim^\phi_I$ do define 
equivalent norms, even though they are not themselves equivalent.

\begin{lem}\label{lem:phiMinusLimi}
If $\phi \in\Phiw(\Omega)$ satisfies \aone{}, then 
there exist a constant $\beta \in (0,1]$ such that
\[
\Volim^\phi_I(f) \leq 1
\implies
\V^\phi_I(\beta f) \leq 1.
\]
Furthermore, $\|\cdot\|_{\Volim^\phi_I} \approx \|\cdot\|_{\V^\phi_I}$.
\end{lem}
\begin{proof}
Let $\epsilon>0$ and choose $\delta>0$ such that 
\[
\sup_{|(I_k)|<\delta}
\sum_{k=1}^n \phi^+_{I_k} \bigg(\frac{|\Delta_kf|}{|I_k|}  \bigg) |I_k|
\le
\Volim^\phi_I(f) + \epsilon
\le 
1+\epsilon.
\]
Let $(I_k)$ be a partition of $I$ and $(I_k')$ be a subpartition of $(I_k)$ with $|(I_k')|<\delta$. 
As in Lemma~\ref{lem:stdGrowthLimit}, we find that 
\[
\phi_{I_k}^- \bigg(\beta \frac{|\Delta f(I_k)|}{|I_k|} \bigg) |I_k|
\le
\sum_{j=j_k}^{j_{k+1}-1}\frac{|I_j'|}{|I_k|} \phi_{I_k}^- \bigg(\frac{|\Delta f(I_j')|}{|I_j'|} \bigg)|I_k|
\le
\sum_{j=j_k}^{j_{k+1}-1} \phi_{I_j'}^- \bigg(\frac{|\Delta f(I_j')|}{|I_j'|} \bigg)|I_j'|,
\]
where we used the quasi-convexity property with constant $\beta$ of $\phi^-$ in the first step 
\cite[Corollary~2.2.2]{HarH19} (see also \cite{HarHJ_pp}).
Thus we conclude that 
\[
\sum_{k=1}^n\phi_{I_k}^- \bigg(\beta \frac{|\Delta f(I_k)|}{|I_k|} \bigg) |I_k|
\le
\sum_{j=1}^{j_{n+1}-1} \phi_{I_j'}^+ \bigg(\frac{|\Delta f(I_j')|}{|I_j'|} \bigg)|I_j'|
\le 
1+\epsilon.
\]
This holds for any partition $(I_k)$ and any $\epsilon>0$, so we deduce that 
\[
\sup_{(I_k)}
\sum_{k=1}^n\phi_{I_k}^- \bigg(\beta \frac{|\Delta f(I_k)|}{|I_k|} \bigg) |I_k|
\le
1.
\]
Since $\phi_{I_k}^-(\beta \frac{|\Delta f(I_k)|}{|I_k|})\le |I_k|^{-1}$, it follows from 
\aone{} that 
\[
\phi_{I_k}^+ \bigg(\beta^2 \frac{|\Delta f(I_k)|}{|I_k|} \bigg)
\le
\phi_{I_k}^- \bigg(\beta \frac{|\Delta f(I_k)|}{|I_k|} \bigg)+1.
\]
Here the second $\beta$ is from \aone{}. Hence
\[
V^\phi_I(\beta^2 f)
=
\sup_{(I_k)}
\sum_{k=1}^n\phi_{I_k}^+ \bigg(\beta^2 \frac{|\Delta f(I_k)|}{|I_k|} \bigg) |I_k|
\le
\sup_{(I_k)}
\sum_{k=1}^n\phi_{I_k}^- \bigg(\beta \frac{|\Delta f(I_k)|}{|I_k|} \bigg) |I_k| + |I|
\le
|I|+1.
\]
Then it follows from \ainc{1} that $V^\phi_I(\frac{\beta^2}{L(|I|+1)} f)\le 1$. 
The claim regarding the norm follows from this and $\Volim^\phi_I\le \V^\phi_I$ 
by homogeneity.
\end{proof}


\section{Connection to other spaces}\label{sect:connections}

A famous result regarding bounded variation functions in the Riesz sense is the Riesz representation lemma, 
which connects the Riesz seminorm with the $L^p$-norm of the derivative of the function. A similar phenomenon 
occurs in the generalized Orlicz case under suitable assumptions. 

We start with approximate equality. 
Note that we do not assume \ainc{} and \adec{}; thus we generalize also the previous results from the 
variable exponent case \cite{CasGR16}. Note that the assumption $f\in \AC(I)$ can be replaced by 
\ainc{}, since together with \azero{} it implies that the function is absolutely continuous. 

\begin{thm}
Let $\phi \in \Phiw(I)$ satisfy \azero{} and \aone{}. If $f\in \AC(I)$ and $f' \in L^\phi(I)$, then $f \in \RBV^\phi(I)$ and
\[
\|f\|_{\RBV^\phi(I)}\le c \|f'\|_{L^\phi(I)}.
\]
\end{thm}

\begin{proof}
Let $\|f'\|_{L^\phi(I)}<1$. Since $f$ is absolutely continuous, 
\[
\frac{|\Delta f(I_k)|}{|I_k|}
=
\bigg|\fint_{I_k} f'\,dx\bigg|
\le
\fint_{I_k} |f'|\,dx.
\]
By \cite[Theorem~4.3.2]{HarH19} there exists $\beta>0$ such that 
\[
\phi^+_{I_k}\bigg(\beta\fint_{I_k} |f'|\,dx\bigg)
\le
\fint_{I_k} \phi(x, |f'|)\, dx+1.
\]
It follows that 
\[
\V^\phi_I(\beta f, (I_k)) \le 
\sum_k \phi^+_{I_k}\bigg(\beta\fint_{I_k} |f'|\,dx\bigg) |I_k|
\le 
\int_I \phi(x, |f'|)\, dx +|I|\le 1+|I|.
\]
Since this holds for any partition $(I_k)$, we find that $\V^\phi_{I}(\beta f)\le 1+|I|$. 
By \ainc{1}, $\V^\phi_{I}(\frac\beta{L(1+|I|)} f)\le \frac1{1+|I|}V^\phi_I(\beta f)\le 1$.
Hence $\|f\|_{\RBV^\phi(I)} \le \frac{L(1+|I|)}\beta$; the claim for general
$\|f'\|_\phi$ follows from this by the homogeneity of the norm. 
\end{proof}

We next derive the corresponding upper bound. Note that here we use only the 
limit $|(I_k)|\to 0$, so the result holds also with $\Volim^\phi_I$ in place of $\V^\phi_I$.

\begin{thm}\label{thm:lower-bound}
Let $\phi \in \Phiw(I)$.
If $f \in \RBV^\phi(I)\cap \AC(I)$, then $f' \in L^\phi(I)$ and 
\[
\|f'\|_{L^\phi(I)} \le \|f\|_{\RBV^\phi(I)}.
\]
\end{thm}
\begin{proof} 
We assume first that $\V^\phi_{I}(f)\le 1$. 
Since $f\in \AC(I)$, the derivative $f'$ exists almost everywhere in $I$. 
Let $((I_k^n)_k)_n$ be a sequence of partitions of $I$ with $|(I_k^n)_k|\to 0$ as $n\to\infty$. 
Define a step-function
\[
F_n:= \sum_k \frac{\Delta f(I_k^n)}{|I_k^n|}\chi_{I_k^n}. 
\]
Since $\lim_n F_n = f'$ a.e.\ and $\phi$ is increasing, we see that $\phi(x, \beta\, |f'(x)|)\le \liminf_n \phi(x, |F_n(x)|)$ for 
a.e.\ $x\in I$ and fixed $\beta\in (0,1)$. Hence Fatou's lemma implies that 
\[
\int_I \phi(x, \beta\, |f'|)\, dx
\le
\liminf_n \int_I \phi(x, |F_n|)\, dx
\le
\liminf_n \sum_k \int_{I_k^n} \phi_{I_k^n}^+(|F_n|)\, dx.
\]
By the definition of $F_n$, 
\[
\int_{I_k^n} \phi_{I_k^n}^+(|F_n|)\, dx
= 
\phi^+_{I_k^n} \bigg(\frac{|\Delta f(I_k^n)|}{|I_k^n|} \bigg)|I_k^n|.
\]
Thus
\begin{align*}
\int_I \phi(x, \beta |f'|)\, dx
\le
\liminf_n \V(f, (I_k^n))
\le
\V^\phi_{I}(f)\le 1.
\end{align*}
This implies that $\|f'\|_{\phi}\le \frac1\beta$ and the general case follows by homogeneity as $\beta\to 1^-$. 
\end{proof}

Combining the previous two results, we obtain the following:
\begin{cor}\label{cor:bound}
Let $\phi \in \Phiw(I)$ satisfy \azero{} and \aone{}.
Then 
\[
\RBV^\phi(I)\cap L^\phi(I)\cap \AC(I) = W^{1,\phi}(I).
\]
\end{cor}

We next derive an exact formula for the Riesz semi-norm. In this case, we have to restrict our attention to the 
$\limsup$-version $\Volim^\phi$. This result has no analogue in \cite{CasGR16}, so it is new even in the 
variable exponent case. 

Following \cite[Section~3.2]{AmbFP00} we consider functions $f$ of bounded variation on the real line 
whose derivative can be described as a signed measure $Df$ with finite total variation, $|Df|(I)<\infty$. 
The measure $Df$  can be split into an absolutely continuous part represented by $f'\,dx$ and a 
singular part $D^sf$ (with respect to the Lebesgue measure). 
In \cite[(3.24)]{AmbFP00}, it is shown that 
\[
\inf_{g=f\text{ a.e.}} \V_I^1(g) =: \EV^1_I(f) = |Df|(I);
\]
the left-hand side is called the \textit{essential variation}. Without the almost everywhere equivalence, 
the equality does not hold since 
we may take a function $f$ equal to zero except at one point so that $V_I^1(f)>0=|Df|(I)$. Functions for which 
the essential variation equals the variation (i.e.\ $\V^1_I=\EV^1_I$) 
are called \textit{good representatives} in \cite{AmbFP00}. 

A left-continuous function of bounded variation is an example of a good representative and can be expressed as $f(x)=Df([a, x)) + f(a)$. Furthermore, by \cite[Theorem~3.28]{AmbFP00}, the left-continuous representative of a 
function of bounded variation is a good representative. For simplicity, we restrict our attention to 
left-continuous functions. Note that \cite{AmbFP00} defined variations 
on open intervals. Where necessary, we can treat the first interval $[a, x]$ separately by a direct calculation. 

Following \cite{EleHH_pp}, we define, for $\phi\in\Phic$,
\[
\phi'_\infty(x):= \lim_{t\to\infty} \frac{\phi(x, t)}t.
\]
Note that the limit exists since $\frac{\phi(x, t)}t$ is increasing. If $\phi$ is differentiable and convex, then $\displaystyle \phi'_\infty(x) = \lim_{t\to\infty} \phi'(x,t)$, hence the notation. 

\begin{thm}[Riesz representation theorem]\label{thm:preciseBV}
Let $f\in \BV(I)$ be left-continuous. If $\phi\in \Phic(I)$ satisfies \vaone{} and \dec{}, then 
\[
\Volim^\phi_{I}(f) = \int_I \phi(x, |f'|)\, dx + \int_I \phi'_\infty \, d|D^sf|. 
\]
\end{thm}
\begin{proof}
We prove the inequality ``$\le$'' and assume that the right-hand side is finite.
By Lemma~\ref{lem:phiMinusLimit}, we may replace $\Volim^\phi_{I}(f)$ by $\underline V^\phi_I$ 
for the lower bound. Let $\epsilon>0$ and choose $\delta\in (0, \epsilon)$ such that 
\[
\int_A \phi(x, |f'|)\, dx < \epsilon
\]
for any set $A\subset I$ with $|A|< \delta$. Since the support of the singular part 
$D^sf$ has measure zero we can choose 
by the definition of the Lebesgue measure 
a finite union $A:=\cup_{i=1}^m [a_i, a_i')$ with $|A|< \delta$ and 
\begin{equation}\label{eq:singular}
\int_{I\setminus A} \phi'_\infty\, d|D^sf| < \epsilon. 
\end{equation}

By left-continuity, $f(x)-f(y)=Df([a, x))-Df([a, y)) = Df([y, x))$ for $x>y$. Thus
\[
\frac{|\Delta f(I_k)|}{|I_k|}
=
\bigg|\fint_{I_k} f'\, dx + \frac{D^sf(I_k)}{|I_k|}\bigg|
\le
\fint_{I_k} |f'|\,d x + \frac{|D^sf|(I_k)}{|I_k|}.
\]
Since $t+s\le \max\{(1+\theta)t, (1+\theta^{-1})s\}$ we obtain by 
\dec{q} that 
\begin{equation}\label{eq:alpha+beta}
\phi^-_B(t+s) 
\le 
(1+\theta)^q \phi^-_B(t) + (1+\theta^{-1})^q\phi^-_B(s). 
\end{equation}
By the previous estimates, Theorem~\ref{thm:jensen} (with $\beta:=\frac1{1+\omega(|I_k|)}$) and $\phi(x, t)\le \phi'_\infty(x)t$ we obtain
\begin{align*}
\phi^-_{I_k}\bigg(\frac{|\Delta f(I_k)|}{|I_k|}\bigg)
& \le
(1+\theta)^q \phi^-_{I_k}\bigg(\fint_{I_k} |f'|\,d x\bigg) 
+ (1+\theta^{-1})^q \phi^-_{I_k}\bigg(\frac{|D^sf|(I_k)}{|I_k|}\bigg)\\
&\le
(\tfrac{1+\theta}\beta)^q \fint_{I_k} \phi^-_{I_k}(|f'|)+\omega(|I_k|)\,dx + (1+\theta^{-1})^q\frac{|D^sf|(I_k)}{|I_k|}(\phi'_\infty)^-_{I_k}.
\end{align*}

We first apply the previous inequality to the set $A_i:=[a_i, a_i')$ from $A$ defined above, 
assumed to be so small that $\omega(|A_i|)\le \epsilon$, 
and choose $\theta:=\epsilon^{-1/(2q)}$:
\begin{align*}
\sum_{i} \phi^-_{A_i}\bigg(\frac{|\Delta f(A_i)|}{|A_i|}\bigg)|A_i| 
&\le
(\tfrac{1+\theta}\beta)^q\int_{A} \phi(x, |f'|)\,d x +\epsilon + (1+\theta^{-1})^q\int_{A} (\phi'_\infty)^-_{I_k}\, d|D^sf|\\
&\qquad\le
(1+\epsilon^{-1/(2q)})^q \beta^{-q}\epsilon(1+|I|) + (1+\epsilon^{1/(2q)})^q\int_I \phi'_\infty\, d|D^sf|.
\end{align*}
Choose sufficiently small complementary closed intervals $B_i=[b_i, b_i')$, i.e.\ $\cup_i A_i + \cup_i B_i=I\setminus\{b\}$ and 
$A_i\cap B_j=\emptyset$, such that $\omega(B_i)\le \epsilon$. We use the same estimate but now choose $\theta:=\epsilon^{1/(2q)}$ 
and use \eqref{eq:singular} to obtain 
\begin{align*}
\sum_{i}
\phi^-_{B_i}\bigg(\frac{|\Delta f(B_i)|}{|B_i|}\bigg)|B_i|
&\le
(1+\epsilon^{1/(2q)})^q\beta^{-q}\int_I \phi(x, |f'|)+\epsilon\,dx  + (1+\epsilon^{-1/(2q)})^q\epsilon.
\end{align*}
Adding these two estimates and letting $\epsilon\to 0$ and $\beta\to 1^-$, we obtain that 
\[
\Volim^\phi_{I}(f) = \Vulim^\phi_{I}(f) \le \int_I \phi(x, |f'|)\, dx + \int_I \phi'_\infty\, d|D^sf|.
\]

For the opposite inequality we use $\Volim^\phi_I$ and start by observing that 
\[
\frac{|\Delta f(B_i)|}{|B_i|}
\ge
\bigg|\fint_{B_i} f'\,d x\bigg| - \frac{|D^sf|(B_i)}{|B_i|}.
\]
This time we set $t=u-s$ in \eqref{eq:alpha+beta} and use the resulting inequality 
\[
\phi^+_B(u-s) \ge
(1+\theta)^{-q} \phi^+_B(u) - \theta^{-q}\phi^+_B(s). 
\]
With $B_i$ as before, we now obtain that 
\begin{align*}
\sum_{i}
\phi^+_{B_i}\bigg(\frac{|\Delta f(B_i)|}{|B_i|}\bigg)|B_i|
&\ge
(1+\theta)^{-q}\sum_{i} \phi^+_{B_i}\bigg(\bigg|\fint_{B_i} f'\,d x\bigg|\bigg)|B_i| 
- \theta^{-q} \epsilon \\
&\ge
(1+\theta)^{-q}\bigg[\int_I\phi(x, |f'|)\, dx - \epsilon \bigg] 
- \theta^{-q} \epsilon,
\end{align*}
where the second inequality follows as in Theorem~\ref{thm:lower-bound} 
(possibly after restricting $\delta$ to a smaller value). 
In the case $\int_I\phi(x, |f'|)\, dx = \infty$, 
we replace the square bracket with $\frac1\epsilon$ and obtain a lower bound tending to $\infty$. Otherwise, we choose $\theta:=\epsilon^{1/(2q)}$ and continue with the estimate of the 
singular part of the derivative. 

Fix $\lambda>1$. Assume that $\delta\le \lambda^{-2q}$ so that $|A|\le \lambda^{-2q}$. 
By absolute continuity of the non-singular part, we may assume $\delta$ is so small that 
\begin{equation}\label{eq:Df}
|Df|(A) \le |D^sf|(A) + \frac1 {\phi^+_I(\lambda)}.
\end{equation}
Since $u$ is a good representative, $V_{A_i}^1(f) = |Df|(A_i)$, which implies that 
\[
\sum_k |\Delta f(A_{i,k})| > |Df|(A_i) - \frac1 {i_{\max}\phi^+_I(\lambda)},
\]
for any subpartition $(A_{i,k})$ of $A_i$ with sufficiently small $\max\{|A_{i,k}|\}$, 
where $i_{\max}$ is the number of intervals $A_i$. 
Now if $\frac{|\Delta f(A_{i,k})|}{|A_{i,k}|}\ge \lambda$, then \inc{1} (from the convexity of $\phi$) 
implies that 
\[
\phi^+_{A_{i,k}}\bigg(\frac{|\Delta f(A_{i,k})|}{|A_{i,k}|}\bigg) 
\ge 
\frac{\phi^+_{A_{i,k}}(\lambda)}{\lambda} \frac{|\Delta f(A_{i,k})|}{|A_{i,k}|}.
\]
If, on the other hand, $\frac{|\Delta f(A_{i,k})|}{|A_{i,k}|}< \lambda$, then 
\[
\phi^+_{A_{i,k}}\bigg(\frac{|\Delta f(A_{i,k})|}{|A_{i,k}|}\bigg)
\ge 0
\ge
\frac{\phi^+_{A_{i,k}}(\lambda)}{\lambda} \frac{|\Delta f(A_{i,k})|}{|A_{i,k}|}
- 
\phi^+_{A_{i,k}}(\lambda). 
\]
Therefore in either case we have the inequality 
\[
\phi^+_{A_{i,k}}\bigg(\frac{|\Delta f(A_{i,k})|}{|A_{i,k}|}\bigg) 
\ge 
\frac{\phi^+_{A_{i,k}}(\lambda)}{\lambda} \frac{|\Delta f(A_{i,k})|}{|A_{i,k}|} 
- \phi^+_I(\lambda).
\]
By \dec{q}, \azero{} and $|A_I|\le \delta\le \lambda^{-2q}$, 
we obtain $\phi^+_{A_i}(\lambda)\le (\lambda/\beta)^q \phi^+_{A_i}(\beta) \le \frac K{|A_i|}$. Hence by \azero{}, \vaone{} 
and $\lambda>1$, we see that 
$\phi^+_{A_i}(\lambda) \le (1+\epsilon)\phi^-_{A_i}(\lambda) \le (1+\epsilon)\phi^+_{A_{i,k}}(\lambda)$.
We conclude that 
\begin{align*}
\sum_{i, k}
\phi^+_{A_{i,k}}\bigg(\frac{|\Delta f(A_{i,k})|}{|A_{i,k}|}\bigg)|A_{i,k}|
&\ge
\sum_i \sum_k
\bigg(\frac{\phi^+_{A_{i,k}}(\lambda)}{\lambda} |\Delta f(A_{i,k})| - \phi^+_I(\lambda)|A_{i,k}| \bigg)\\
&\ge
\sum_i 
\frac{\phi^+_{A_{i}}(\lambda)}{(1+\epsilon)\lambda}  \bigg(|Df|(A_i) - \frac1 {i_{\max}\phi^+_I(\lambda)} \bigg) - \phi^+_I(\lambda)|A| \\
&\ge
\sum_{i}
\frac{\phi^+_{A_i}(\lambda)}{(1+\epsilon)\lambda} |Df|(A_i)- \frac1\lambda - \lambda^{-2q}\phi^+_I(\lambda).
\end{align*}
Again, by \dec{q} and \azero{}, $\lambda^{-2q}\phi^+_I(\lambda) \le \lambda^{-q}\phi^+_I(1)\le \frac c\lambda$.

Next we observe by \eqref{eq:Df} that
\begin{align*}
\sum_{i} \frac{\phi^+_{A_i}(\lambda)}{\lambda} |Df|(A_i)
&\ge 
\sum_{i} \frac{\phi^+_{A_i}(\lambda)}{\lambda} |D^sf|(A_i)
-
\frac{\phi^+_I(\lambda)}{\lambda} \sum_{i}\big[|Df|(A_i)-|D^sf|(A_i)\big]\\
& \ge
\sum_{i} \frac{\phi^+_{A_i}(\lambda)}{\lambda} |D^sf|(A_i) - \frac 1\lambda
=
\int_I \sum_{i} \frac{\phi^+_{A_i}(\lambda)}{\lambda} \chi_{A_i} \, d|D^sf| - \frac 1\lambda.
\end{align*}
Since $\phi^+_{A_i}(\lambda)\ge \phi(x,\lambda)$ for every $x\in A_i$, we obtain that 
\begin{align*}
\sum_{i,k}
\phi^+_{A_{i,k}}\bigg(\frac{|\Delta f(A_{i,k})|}{|A_{i,k}|}\bigg)|A_{i,k}|
&\ge
\int_A \frac{\phi(x, \lambda)}{(1+\epsilon)\lambda}\, d|D^sf| - \frac c\lambda.
\end{align*}
Finally, we combine the estimate over $A$ and $B$ and have thus show that
\[
\Volim^\phi_{I}(f) 
\ge 
(1+\epsilon^{1/(2q})^{-q} \int_I \phi(x, |f'|)\, dx + 
\int_A \frac{\phi(x, \lambda)}{(1+\epsilon)\lambda}\, d|D^sf| 
 - \frac c\lambda - c\sqrt\epsilon. 
\]
We obtain the desired lower bound as $\epsilon\to 0$ and $\lambda\to \infty$ by monotone convergence, 
since $\frac{\phi(x, \lambda)}{\lambda} \nearrow \phi'_\infty(x)$, also using \eqref{eq:singular}. 
\end{proof}

In the situation of the previous theorem, we precisely regain the $\phi$-norm of the derivative of 
the function. Note that this result is new also in the variable exponent case. 
Furthermore, Example~\ref{eg:limitDifferent} shows that the result does not hold for 
$\V^\phi_I$. 

\begin{cor}\label{cor:preciseW}
Let $f\in \AC(I)$. If $\phi\in \Phic(I)$ satisfies \vaone{} and \dec{}, then 
\[
\Volim^\phi_{I}(f) = \int_I \phi(x, |f'|)\, dx.
\]
\end{cor}

We conclude by commenting on the corollaries from the introduction. 
If $p:\Omega\to [1,\infty)$ is a bounded variable exponent, then \vaone{} is equivalent to
the strong $\log$-Hölder condition. Further, $\phi'_\infty = 1+\infty \chi_{\{p>1\}}$ so that 
\[
\int_I \phi'_\infty\, d|D^sf| = \int_I 1+\infty \chi_{\{p>1\}}\, d|D^sf|
=|D^sf|(\{p=1\})
\]
when $\chi_{\{p>1\}}=0$ $|D^sf|$-a.e. 
In the double phase case $\phi(x,t)=t +a(x)t^q$, we similarly obtain 
$\phi'_\infty = 1+\infty \chi_{\{a>0\}}$, and the corollary follows. 
Finally, in the Orlicz case $\phi'_\infty$ is a constant. If the constant is infinity, 
then the singular part must vanish in order that $\infty |D^sf|(I)$ be finite, and so the 
function is absolutely continuous.


\bigskip

 \noindent\small{\textsc{P. Hästö}\\ 
Department of Mathematics and Statistics,
FI-20014 University of Turku, Finland}\\
\footnotesize{\texttt{peter.hasto@utu.fi}}\\

\noindent\small{\textsc{J. Juusti}}\\
\small{Department of Mathematics and Statistics,
FI-20014 University of Turku, Finland}\\
\footnotesize{\texttt{jthjuu@utu.fi}}\\

\noindent\small{\textsc{H. Rafeiro}}\\
\small{Department of Mathematical Sciences,
United Arab Emirates University, College of Science, UAE}\\
\footnotesize{\texttt{rafeiro@uaeu.ac.ae}}\\

\newpage
\appendix

\end{document}